\documentclass[11pt]{article}
\usepackage{amssymb,amsmath,accents}
\usepackage{amscd}
\usepackage{amsfonts,amsthm,mathrsfs}
\usepackage{setspace}
\usepackage{cases,empheq}
\usepackage{subcaption}
\usepackage{diagbox}
\usepackage[hyphens,allowmove]{url}
\usepackage{hyperref}
\usepackage{cleveref}
\usepackage{float}

\numberwithin{equation}{section}
\newtheorem{thm}{Theorem}
\newtheorem{cor}[thm]{Corollary}
\newtheorem{lem}[thm]{Lemma}
\newtheorem{prop}[thm]{Proposition}

\newtheorem{definition}[thm]{Definition}
\newtheorem{rem}[thm]{Remark}

\allowdisplaybreaks[4]

 \newcommand{\p}{\partial}
\newcommand{\norm}[1]{\left\lVert #1 \right\rVert}

\newcommand{\R}{\mathbb{R}}

\usepackage{enumitem}

\begin{document}

\title
{Solutions to a One-Dimensional Combustion-Type Free Boundary Problem via Maximal Regularity}

\author{Ken Furukawa$^1$, Yoshikazu Giga$^2$ and Naoto Kajiwara$^3$ \vspace{0.7em}\\ 
\centerline{{\small $^1$Faculty of Science, Academic Assembly}} \\
\centerline{{\small University of Toyama}} \\
\centerline{{\small 3190 Gofuku, Toyama-shi, Toyama 930-8555, Japan}} \vspace{0.7em}\\
\centerline{{\small $^2$Graduate School of Mathematical Sciences}} \\
\centerline{{\small The University of Tokyo}} \\
\centerline{{\small 3-8-1 Komaba, Meguro-ku, Tokyo 153-8914, Japan}} \vspace{0.7em}\\
\centerline{{\small $^3$Applied Physics Course, Department of Electrical, Electronic and Computer Engineering}} \\
\centerline{{\small Gifu University}} \\
\centerline{{\small 1-1 Yanagido, Gifu City, Gifu 501-1193, Japan}}}

\date{}

\maketitle
\thispagestyle{empty}
%

\begin{abstract}
We study a one-dimensional free boundary problem arising in combustion theory, where the motion of the interface is governed by a prescribed Neumann boundary flux and a zero Dirichlet boundary condition.
We treat both the half-line case and the bounded interval case.
For both settings, we employ maximal $L^p$-$L^q$ regularity as our main analytical tool.
In the half-line case, the solutions need not decay at infinity, even though the spatial derivatives belong to $L^q(\R_+)$. 
To handle the evolution law of the free boundary, we introduce a derivative formulation that avoids second-order boundary traces.
By combining maximal $L^p$-$L^q$ regularity and Schauder estimates, we establish the local-in-time existence, uniqueness, and regularity of solutions, as well as the evolution law of the free boundary. 
\end{abstract}

\maketitle

\section{Introduction}\label{sec:intro}
Free boundary problems of Stefan and combustion-type arise in phase transitions, flame propagation, and reaction-diffusion systems.
In this paper, we study a one-dimensional free boundary problem in which the motion of the interface is determined subject to a prescribed Neumann flux and a homogeneous Dirichlet condition. 
This problem may be regarded as a one-dimensional parabolic Bernoulli-type free boundary problem. 
See, e.g., \cite{BL,CV,Fried}. 
Mathematically, the main difficulty lies in the strong coupling between the parabolic equation and the evolution law for the free boundary.

We study the following free boundary problem. 
\begin{align}\label{eq:u}
	\begin{cases}
		\p_t u - \p_x^2 u = f(x,t), & s(t)<x<\infty,\ t>0,\\
		u(s(t),t)=0, & t>0,\\
		\p_x u(s(t),t)=\alpha, & t>0,\\
		u(x,0)=u_0(x), & x>0,\\
		s(0)=0,
	\end{cases}
\end{align}
where $\alpha>0$ and $f$ is a given heat source.
Here the free boundary $s(t)$ represents the interface determined by the prescribed boundary conditions, which are typical in combustion-type models. 
One of the main objectives of this paper is to establish a local-in-time existence and uniqueness theorem for problem \eqref{eq:u} under suitable compatibility conditions on the initial data. 
For the reader's convenience, we first state a simplified version of our main result. 
The full statement will be given in Theorem \ref{thm:main}. 
We denote by $\mathcal{S}(\R)$ the Schwartz space of rapidly decreasing smooth functions on $\R$. 
\begin{thm}
Assume that $f : \R\times[0,T] \to \R$ is bounded and H\"older continuous, and that $\p_x f(\cdot,t) \in \mathcal{S}(\R)$ uniformly for $t\in[0,T]$. 
Let $w_0 \in \mathcal{S}(\R)$ satisfy $w_0(0)=\alpha$, and define 
\[u_0(x)=\int_0^x w_0(y)dy.\]
Then the problem \eqref{eq:u} admits a unique classical solution $(u,s)$ on a time interval $(0,T^\ast)$ for some $T^\ast\in(0,T]$.
\end{thm}
\noindent
From a physical viewpoint, we prove that $u$ is positive when $f$ and $u_0$ are positive by the maximum principle. 
This is stated in Remark \ref{rem:positivity}. 

The present paper can be regarded as a continuation of our previous work \cite{FGK}, but it explores a different analytical direction. 
In \cite{FGK}, we focused on the effect of incompatible initial data and bounded external forces $f$, and in particular we proved a non-existence result for the case $u_0 \equiv 0$ under bounded forcing. 
We also constructed a unique bounded self-similar solution on the half-line when the external force is of the form $f=h/\sqrt{t}$, where $h$ is a constant. 
The analysis in \cite{FGK} was carried out on the half-line and a bounded interval as well as special multi-dimensional domains. 
In contrast, the present paper aims to study the local-in-time well-posedness for compatible initial data by a completely different analytical approach in one-dimensional space. 
We treat both the half-line case and the case of a bounded interval with two free boundaries. 
The theorem for a bounded interval is given in Theorem \ref{thm:main2}. 
The proofs of the these two theorems are different, and the latter admits a simpler proof. 
In the half-line setting, we allow solutions that do not necessarily decay at spatial infinity, although their first spatial derivatives are required to decay. 

\subsection*{Strategy of the paper}
To analyze the problem, we first transform the free boundary problem into a fixed-domain formulation by introducing a moving coordinate system, i.e., $y=x-s(t)$ and $v(y,t):=u(s(t)+y,t)$. 
Formally taking the trace of the transformed equation at the moving boundary leads to an evolution equation for the free boundary:
\[\alpha s'(t)= -\p_y^2 v(0,t)-f(s(t),t).\]
However, this procedure leads to a {\it second-order} boundary trace, which is not compatible with the standard maximal $L^p$ regularity framework, even for linear parabolic equations. 

To overcome this difficulty, we reformulate the problem in terms of the derivative 
\[w = \p_y v\]
where $v$ is the transformed unknown function above. 
This reformulation eliminates the second-order boundary trace and replaces it with a system consisting of
\begin{itemize}
\item a second-order parabolic equation for $w$, and
\item an ordinary differential equation for the free boundary $s$, with a first-order trace of $w$.
\end{itemize}
See the equation \eqref{eq:w} in the following section. 
Here, the equation for $w$ is obtained by differentiating the equation for $v$ with respect to $y$. 
The key observation is that the resulting system involves only lower-order boundary traces which are compatible with the maximal $L^p$-$L^q$ regularity framework. 
This allows us to apply a fixed point argument in appropriate function spaces. 

Our analysis on the half-line proceeds in three steps:
\begin{itemize}
\item[Step\,1] We establish local well-posedness of the derivative system for $(w,s)$ in $L^p$-based spaces using maximal regularity.
\item[Step\,2]  For the obtained free boundary $s(t)$, we solve a {\it linear} parabolic problem for $v$ using Schauder theory.
\item[Step\,3]  We identify $w = \partial_y v$ by a uniqueness argument and recover a classical solution to the original problem. 
\end{itemize}
In the first step we use the fixed point theorem. 
To estimate the nonlinear terms, we use two different $L^p$ spaces with respect to time. 
As for Steps\,2 and 3, in the case of a bounded interval with two free boundaries, the recovery of $v$ from $w$ can be achieved by indefinite integration, and Schauder theory is not needed at all. 
This provides a new functional analytic framework for the problem, which avoids the limitations of direct approaches based on maximal regularity. 

An alternative approach would be to work in a fractional Sobolev framework. 
However, we need to carefully track the time dependence throughout the calculations. 
This becomes technically involved when dealing with trace operators and embeddings in fractional Sobolev spaces.
For these reasons, we do not adopt a fractional Sobolev approach. 
\subsection*{Related literature}
We review related works. 
There are so many related works, so our review is not at all exhaustive. 
More related articles are found in a review paper \cite{GHV} or a recent paper \cite{FGK}. 

Combustion-type free boundary problems arise as high activation energy limits of reaction--diffusion equations. 
In this limit, the reaction zone becomes a moving interface, and the temperature satisfies the heat equation away from the interface. 
The classical model corresponds to the homogeneous case $f\equiv0$. 
In this paper, we consider the more general problem with a forcing term $f$, which represents external heat sources or sinks. 
In most of the references mentioned below, the homogeneous case is considered. 
Non-homogeneous problems have been studied only in limited settings, including the authors' works \cite{FGK, FGK2}. 
Caffarelli and V\'azquez \cite{CV} studied a multidimensional combustion-type free boundary problem in the homogeneous case $f=0$. 
They established existence of weak solutions together with several qualitative properties of the free boundary. 

Classical solutions for multidimensional combustion-type free boundary problems with $f=0$ were investigated by Andreucci and Gianni \cite{AG}. 
The problem is also considered as a singular limits of two-phase parabolic problems; see Caffarelli, Lederman, and Wolanski \cite{CLW, CLW2}. 
For the physical background and mathematical modeling of combustion phenomena, we refer to the lecture notes by Buckmaster and Ludford \cite{BL} and Williams \cite{W}. 

Early studies on free boundary problems for the heat equation go back to the work of Ventcel' \cite{V}. 
Free boundary problems of Stefan or combustion-type have been extensively discussed in the book of Friedman \cite{Fried}. 
Here, the Stefan boundary condition refers to $s'=-\partial_x u(s(t),t)$, in contrast to the prescribed flux condition considered in this paper. 
In the Stefan problem, the interface velocity is given explicitly by the first derivative of the function $u$. 
In contrast, for combustion-type problems the interface velocity is determined only implicitly through the solution of the parabolic equations. 
As shown above, the evolution law for the free boundary involves a second-order boundary trace of the solution. 
The boundary condition $\partial_x u=\alpha$ can be viewed as a one-dimensional counterpart of the Bernoulli-type condition $|\nabla u|=\alpha$ in higher dimensions. 

In the elliptic setting, problems (with $f=0$) with a related structure are known as Bernoulli free boundary problems. 
Elliptic-parabolic free boundary problems arising in combustion theory have been extensively studied by Hulshof and collaborators, including results on interface continuity \cite{H1}, spherical symmetry and convergence to traveling waves \cite{HH, H2}.  
Focusing phenomena were further analyzed by Hilhorst and Hulshof \cite{HH2}. 
Other approaches based on weak solutions can be found in \cite{CV, K}. 
A planar travelling wave solution exists, and its nonlinear stability was established by Brauner, Hulshof and Lunardi \cite{BHL}. 
Their analysis relies on a suitable decomposition of the solution together with optimal decay estimates and the Cole-Hopf transformation. 
Extinction and self-similar behavior of spherical and annular flames were discussed in detail by Galaktionov, Hulshof and V\'azquez \cite{GHV} which includes history of researches. 
Most of these works are concerned with well-posedness, qualitative properties of interfaces, or asymptotic behavior of particular solutions when $f=0$, while our paper treats strong solutions based on maximal $L^p$ regularity with a general forcing term $f$. 

In a forthcoming paper \cite{FGK2}, by introducing a mild solution framework for the spatial derivative of the solution, the authors establish local well-posedness under minimal regularity assumptions on both the forcing term and the initial data. 
They also introduce a derivative system and allow solutions that do not necessarily decay at spatial infinity.
Moreover, the interface velocity satisfies the property that $t^{1/2}s'(t)$ remains bounded as $t\to+0$, where the exponent $1/2$ is consistent with the natural parabolic scaling of the problem. 
\subsection*{Organization of the paper}
The paper is organized as follows.
In Section \ref{sec:reformation}, we reformulate the original free boundary problem on a fixed domain and introduce a derivative formulation.
Section \ref{sec:w} is devoted to the local well-posedness of the derivative system in $L^p$-based spaces.
In Section \ref{sec:reg}, we consider the solution in H\"older spaces again by means of Schauder theory.
Section \ref{sec:v} establishes the identification $w=\partial_y v$ and we recover a classical solution to the original free boundary problem. 
Specifically, Section \ref{sec:w}, \ref{sec:reg} and \ref{sec:v} correspond to the above Steps\,1,2,3, respectively. 
Finally, in Section \ref{sec:bounded}, we study the case of a bounded interval. 
The main result is the existence and uniqueness theorem for the problem. 

\section{Reformulation of the Free Boundary Problem}\label{sec:reformation}

\subsection{Fixing the domain}
Define
\[v(y,t):=u(s(t)+y,t), \qquad y>0.\]
We have $\p_y v=\p_x u$, $\p_y^2 v= \p_x^2 u$ and  $\p_t v=\p_t u + s'\p_x u$. 
Formally differentiating the boundary condition $u(s(t),t)=0$ with respect to $t$, we obtain 
\[\p_t u + s' \p_x u(s(t),t)=0.\]
Hence, 
\[\alpha s'+\p_y^2 v(0,t) + f(s(t),t)=0.\]
Therefore the pair of functions $(v,s)$ satisfies
\begin{equation}\label{eq:v}
	\begin{cases}
		\p_t v - \p_y^2 v = s'(t)\p_y v + f(s(t)+y,t), &y>0,\ t>0,\\
		\alpha s'(t)= -\p_y^2 v(0,t)-f(s(t),t), &t>0,\\
		\p_y v(0,t)=\alpha,&t>0,\\
		v(y,0)=v_0(y),& y>0,\\
		s(0)=0, 
	\end{cases}
\end{equation}
where we have set $v_0(y):=u_0(y)$. 
Taking the trace at $y = 0$ yields $\p_t v(0,t) = 0$, which suggests that $v(0,t) = 0$ provided $v_0(0) = 0$. 
However, at this stage, the operation is only formal, since the trace $\p_y^2 v(0,t)$ is not a priori well-defined in the maximal $L^p$-$L^q$ setting. 
This fact will be rigorously justified later after we prove that $v$ enjoys sufficient regularity at the boundary.
More precisely, we adopt Schauder theory and obtain H\"older regularity. 
Therefore we recover $u(s(t), t) = v(0, t) = 0$. 
To do so, we need to specify the function spaces precisely and we need to solve the problem \eqref{eq:v}. 
We reformulate the problem in terms of $w(=\p_y v)$ in the next subsection and we solve the system for $(w,s)$ in the next section. 

\subsection{Derivative formulation}
The main difficulty in treating \eqref{eq:v} by maximal $L^p$-$L^q$ regularity is the appearance of the second-order boundary trace $\p_y^2 v(0,t)$ in the free boundary condition. 
Such a term is not directly compatible with the standard theory. 

To overcome this difficulty, we introduce a derivative formulation by setting 
\[w:=\p_y v.\]
Then, by letting $w_0(y)=\p_y v_0(y)$, the function $w$ and $s$ satisfies 
\begin{equation}\label{eq:w}
	\begin{cases}
		\p_t w - \p_y^2 w = s'(t)\p_y w + \p_x f(s(t)+y,t), &y>0,\ t>0,\\
		\alpha s'(t)= -\p_y w(0,t)-f(s(t),t), &t>0,\\
		w(0,t)=\alpha,&t>0,\\
		w(y,0)=w_0(y),&y>0,\\
		s(0)=0, 
	\end{cases}
\end{equation}
which is obtained by differentiating the first equation in \eqref{eq:v} with respect to $y$.  
Here we used the identity $\p_y f(s(t)+y,t)=\p_x f(s(t)+y, t)$. 
We note that, at this stage, the equivalence between \eqref{eq:v} and \eqref{eq:w} has not yet been justified because we do not specify the function spaces. 
However, we establish solvability of \eqref{eq:w} in the maximal $L^p$-$L^q$ framework in the next section. 
We emphasize that the boundary term $\p_y w(0,t)$ involves only first-order trace, which is well-defined in this  framework and thus compatible with the standard maximal regularity theory. 
The identification $w=\partial_y v$ will be rigorously established in Section \ref{sec:v}.

\section{Local Well-posedness for the derivative equation \eqref{eq:w}} \label{sec:w}
In this section, we establish local well-posedness for the derivative system \eqref{eq:w}. 
The key point is the use of two different $L^p$ spaces in time. 
The system is viewed as a coupled parabolic-ODE system, where the nonlinear terms are of lower-order and can be controlled by suitable product and trace estimates.
This structure allows us to apply the Banach fixed point theorem in appropriate $L^p$-based spaces.

We begin by introducing the definitions of function spaces and some estimates. 
Let $1<p_1,p_2,q<\infty$ and $T\in(0,\infty)$, $\R_+:=(0,\infty)$. 
We define
\begin{align*}
F_w&:=L^{p_1}(0,T;L^q(\R_+)), \\
E_w&:=W^{1,p_1}(0,T;L^q(\R_+))\cap L^{p_1}(0,T;W^{2,q}(\R_+)), \\
X_w&:=B^{2(1-1/p_1)}_{q,p_1}(\R_+), \\
F_s&:=L^{p_2}(0,T), \quad E_s:=W^{1,p_2}(0,T).
\end{align*}
When necessary, we explicitly write $F_w(0,T)$, $E_w(0,T)$, $F_s(0,T)$ and $E_s(0,T)$ to indicate the time interval explicitly. 
The exponent $p_1$ is used for the parabolic part, and $p_2$ is used for the ODE. 
The space $X_w$ is chosen as the natural trace space associated with $E_w$ in the framework of maximal $L^p$-$L^q$ regularity.
See \cite{PS}. 

Throughout this paper, we use the following exponent condition: 
\begin{center}
{\rm(A)}\quad Let $1<p_1,p_2, q<\infty$ satisfy $2/p_1+1/q<1$ and $p_1<p_2$. 
\end{center}
This ensures the embedding $X_w\subset BUC^1[0,\infty)$ and the boundary value of any function in $X_w$ at $y=0$. 

The following proposition gives the key estimates needed to control the nonlinear coupling terms in the fixed point argument.
\begin{prop}[Product and trace estimates]\label{prop:product}
Assume the condition {\rm (A)}. 
Then there exists a constant $C>0$, depending only on $p_1, p_2$ and $q$, such that the following estimates hold for all $T>0$. 
\begin{enumerate}[label=\upshape(\roman*)]
\item Product and trace estimates. \\
Let $\varphi\in E_w$ and $\xi\in E_s$. 
Then 
\begin{align*}
\norm{\xi'\p_y\varphi }_{F_w}
	&\le C T^{\frac{1}{p_1}-\frac{1}{p_2}}\bigl(\norm{\varphi}_{E_w}+\norm{\varphi|_{t=0}}_{X_w}\bigr)\norm{\xi}_{E_s},\\
\norm{\p_y\varphi|_{y=0}}_{F_s}
	&\le C T^{\frac{1}{p_2}}\bigl(\norm{\varphi}_{E_w}+\norm{\varphi|_{t=0}}_{X_w}\bigr).
\end{align*}
 \item Difference estimates (contraction properties).\\
Let $\varphi_1, \varphi_2\in E_w$ with $\varphi_1|_{t=0}=\varphi_2|_{t=0}$ and $\xi_1, \xi_2\in E_s$.  
Then 
\begin{align*}
\norm{\xi'_1\p_y\varphi_1-\xi'_2\p_y\varphi_2}_{F_w}
	&\le C T^{\frac{1}{p_1}-\frac{1}{p_2}} \left\{\bigl(\norm{\varphi_1}_{E_w}+\norm{\varphi_1|_{t=0}}_{X_w}\bigr)\norm{\xi_1-\xi_2}_{E_s}\right.\\
	&\qquad \qquad  \left.+ \norm{\varphi_1-\varphi_2}_{E_w}\norm{\xi_2}_{E_s}\right\},\\
\norm{\p_y(\varphi_1-\varphi_2)|_{y=0}}_{F_s}
	&\le C T^{\frac{1}{p_2}}\norm{\varphi_1-\varphi_2}_{E_w}.
\end{align*}
\end{enumerate}
\end{prop}

\begin{proof}
By H\"older's inequality, we obtain
\begin{align*}
&\norm{\xi' \p_y \varphi}_{F_w}\\
	\le &C\norm{\varphi}_{L^\infty(0,T; W^{1,q}(\R_+))}\norm{\xi'}_{L^{p_1}(0,T)} \\
	\le &C T^{\frac{1}{p_1}-\frac{1}{p_2}}\Bigl(\norm{\varphi - e^{t\Delta_D}\varphi|_{t=0}}_{L^\infty(0,T; W^{1,q}(\R_+))}+\norm{e^{t\Delta_D}\varphi|_{t=0}}_{L^\infty(0,T; W^{1,q}(\R_+))}\Bigr)\norm{\xi}_{E_s},
\end{align*}
where $e^{t\Delta_D}$ denotes the analytic semigroup generated by the Dirichlet Laplacian on $L^q(\R_+)$.

We recall the trace theorem 
\begin{align*}
E_w&=W^{1,p_1}(0,T; L^q(\R_+))\cap L^{p_1}(0,T; W^{2,q}(\R_+))\\
&\hookrightarrow BUC\bigl([0,T]; B^{2(1-1/p_1)}_{q,p_1}(\R_+)\bigr)=BUC\bigl([0,T];X_w\bigr),
\end{align*}
and that the embedding constant is independent of $T$ provided that the initial value at $t=0$ vanishes.
Therefore,
\begin{align*}
&\norm{\xi' \p_y \varphi}_{F_w}\\
	\le &C T^{\frac{1}{p_1}-\frac{1}{p_2}}\Bigl(\norm{\varphi - e^{t\Delta_D}\varphi|_{t=0}}_{E_w}+\norm{e^{t\Delta_D}\varphi|_{t=0}}_{L^\infty(0,T; W^{1,q}(\R_+))}\Bigr)\norm{\xi}_{E_s} \\
	\le &C T^{\frac{1}{p_1}-\frac{1}{p_2}}\Bigl(\norm{\varphi}_{E_w}+\norm{\varphi|_{t=0}}_{X_w}\Bigr)\norm{\xi}_{E_s},
\end{align*}
where we have used the embedding $X_w\subset W^{1,q}(\R_+)$ and the boundedness of the semigroup $e^{t\Delta_D}$.
For properties of $e^{t\Delta_D}$, see \cite{DHP, PS}. 

Similarly, we estimate the boundary term.
By the embedding $X_w\subset BUC^1(\overline{\R_+})$ again, we have
\begin{align*}
\norm{\p_y \varphi|_{y=0}}_{F_s}
 	&\le C T^{\frac{1}{p_2}}\norm{\varphi}_{L^\infty(0,T; X_w)} \\
	&\le C T^{\frac{1}{p_2}}\Bigl(\norm{\varphi}_{E_w}+\norm{\varphi|_{t=0}}_{X_w}\Bigr).
\end{align*}
The second difference estimates can be proved in the same manner, since $\varphi_1|_{t=0}=\varphi_2|_{t=0}$.
\end{proof}
The factors involving powers of $T$ arise from time integration and play a crucial role in the contraction argument below. 
In order to control the non-homogeneous term $f$ in maximal regularity class, we use the following condition. 
\begin{center}
{\rm (F)}\quad Let $f:\R\times[0,T]\to\R$ satisfy $\begin{cases}
\hspace{4mm}f\in L^{p_2}(0,T;L^\infty(\R)), \\
\p_x f\in L^{p_1}(0,T;L^q(\R))\cap L^\infty(0,T;L^\infty(\R)), \\
\p_x^2 f \in L^\infty(0,T;L^q(\R)). 
\end{cases}$
\end{center}
The assumption on $\p_x^2 f$ is used only to control the composition $\p_x f(s(t)+y, t)$ in the contraction argument. 

\begin{thm}[Local existence for $(w,s)$]\label{thm:w}
Assume the conditions {\rm (A)} and {\rm (F)}. 
Let 
\[w_0\in X_w, \quad w_0(0)=\alpha.\]
Then there exists $T^\ast > 0$ such that the derivative system \eqref{eq:w} admits a unique local-in-time solution 
\[(w, s)\in E_w(0,T^\ast)\times E_s(0,T^\ast)\]
satisfying 
\[\norm{w}_{E_w(0,T^\ast)} + \norm{s}_{E_s(0,T^\ast)} \leq  C_{T^\ast, f, \alpha} + C\norm{w_0}_{X_w}\]
for some constant $C_{T^\ast, f, \alpha}$ and $C$. 
\end{thm}

\begin{proof}
We apply the Banach fixed point theorem in the space $E_w(0,T)\times E_s(0,T)$. 
To this end, we introduce the closed subset 
\begin{equation*}
B_R(T) := \left\{(w,s)\in E_w(0,T)\times E_s(0,T)~\middle|~
  \begin{aligned} 
&w|_{y=0}=\alpha, \\
&w|_{t=0}=w_0,~s(0)=0,\\
&\norm{(w,s)}_{E_w(0,T)\times E_s(0,T)}\le R
  \end{aligned}
  \right\}
\end{equation*}
with $T>0$ and $R>0$ to be fixed later, and the solution mapping 
\begin{align*}
\begin{array}{rccc}
\mathcal{S}:&B_R(T)&\longrightarrow&E_w(0,T) \times E_s(0,T)\\
        & \rotatebox{90}{$\in$}&               & \rotatebox{90}{$\in$} \\
        &(\varphi,\xi)& \longmapsto   & (w,s)
\end{array}
\end{align*}
to
\begin{equation}\label{eq:w2}
	\begin{cases}
		\p_t w - \p_y^2 w= \xi'(t)\p_y \varphi+ \partial_x f(\xi(t)+y,t), & y>0,\ 0<t<T, \\
		\alpha s'(t)= -\p_y \varphi(0,t) - f(\xi(t),t),& 0<t<T, \\
		w(0,t)= \alpha,&0<t<T, \\
		w(y,0)= w_0(y),& y>0, \\
		s(0)= 0.&
\end{cases}
\end{equation}
This system is linear in $(w,s)$ for fixed $(\varphi, \xi)$. 
Proposition \ref{prop:product}, together with the maximal regularity for the heat equation and the ODE, implies that
\begin{align*}
&\norm{\mathcal{S}(\varphi,\xi)}_{E_w(0,T)\times E_s(0,T)}\\
&=\norm{w}_{E_w(0,T)}+\norm{s}_{E_s(0,T)} \\
&\le C\Bigl\{T^{\frac1{p_1}-\frac1{p_2}}\bigl(\norm{\varphi}_{E_w(0,T)}+ \norm{w_0}_{X_w}\bigr)\norm{\xi}_{E_s(0,T)} \\
&\qquad+ \norm{\p_x f}_{L^{p_1}(0,T;L^q(\R))}+\alpha+\norm{w_0}_{X_w} \\
&\qquad+ |\alpha|^{-1}\Bigl(T^{\frac{1}{p_2}}(\norm{\varphi}_{E_w(0,T)}+\norm{w_0}_{X_w})+ \norm{f}_{L^{p_2}(0,T;L^\infty(\R))}\Bigr)\Bigr\}.
\end{align*}
Choosing $T>0$ sufficiently small and $R>0$ sufficiently large (depending on the data), we obtain
\[\|\mathcal{S}(\varphi,\xi)\|_{E_w(0,T)\times E_s(0,T)} \le R\]
whenever $\norm{\varphi}_{E_w(0,T)}+\norm{\xi}_{E_s(0,T)}\le R$, where the constant $R$ depends only on $f$, $\alpha$, and $\norm{w_0}_{X_w}$.
This implies that $\mathcal{S}(B_R(T))\subset B_R(T)$. 

Similarly, we obtain the contraction estimate; for any $(\varphi_i, \xi_i)\in B_R(T)~(i=1,2)$, we have 
\begin{align*}
&\norm{\mathcal{S}(\varphi_1,\xi_1)-\mathcal{S}(\varphi_2,\xi_2)}_{E_w(0,T)\times E_s(0,T)} \\
&\le C T^{\frac1{p_1}-\frac1{p_2}}\Bigl\{(\norm{\varphi_1}_{E_w(0,T)}+\norm{w_0}_{X_w})\norm{\xi_1-\xi_2}_{E_s(0,T)}+ \norm{\varphi_1-\varphi_2}_{E_w(0,T)}\norm{\xi_2}_{E_s(0,T)}\Bigr\} \\
&\quad+ C\Bigl\{\norm{\p_x f(\xi_1(t)+y,t)-\p_x f(\xi_2(t)+y,t)}_{F_w(0,T)} \\
&\quad+|\alpha|^{-1}\Bigl(T^{\frac{1}{p_2}}\norm{\varphi_1-\varphi_2}_{E_w(0,T)}+ \norm{f(\xi_1(t),t)-f(\xi_2(t),t)}_{F_s(0,T)}\Bigr)\Bigr\}\\
&\le C T^{\frac1{p_1}-\frac1{p_2}}\Bigl\{(\norm{\varphi_1}_{E_w(0,T)}+\norm{w_0}_{X_w})\norm{\xi_1-\xi_2}_{E_s(0,T)}+ \norm{\varphi_1-\varphi_2}_{E_w(0,T)}\norm{\xi_2}_{E_s(0,T)}\Bigr\} \\
&\quad+ C\Bigl\{T^{\frac1{p_1}-\frac1{p_2}} \norm{\p_x^2 f}_{L^\infty(0,T;L^q(\R))}\norm{\xi_1-\xi_2}_{E_s(0,T)} \\
&\quad+|\alpha|^{-1}\Bigl(T^{\frac{1}{p_2}}\norm{\varphi_1-\varphi_2}_{E_w(0,T)}+ T\norm{\p_x f}_{L^\infty(0,T;L^\infty(\R))}\norm{\xi_1-\xi_2}_{E_s(0,T)}\Bigr)\Bigr\}.
\end{align*}
Here we have used the estimates 
\begin{align*}
&\norm{\p_x f(\xi_1(t)+y,t)-\p_x f(\xi_2(t)+y,t)}_{F_w(0,T)}
	\le \norm{\p_x^2 f}_{L^\infty(0,T; L^q(\R))}\norm{\xi_1-\xi_2}_{L^{p_1}(0,T)}, \\
&\norm{f(\xi_1(t),t)- f(\xi_2(t),t)}_{F_s(0,T)}
	\le \norm{\p_x f}_{L^\infty(0,T; L^\infty(\R))}\norm{\xi_1-\xi_2}_{F_s(0,T)}, \\
&\norm{\xi_1-\xi_2}_{F_s(0,T)}
	\le T\norm{\xi_1'-\xi_2'}_{L^{p_2}(0,T)}
	\le T\norm{\xi_1-\xi_2}_{E_s(0,T)},
\end{align*}
which follows from Poincar\'e-type inequality for functions vanishing at $t=0$, i.e., $\xi_1(0)=\xi_2(0)$. 
Hence, choosing $T>0$ sufficiently small, the mapping $\mathcal{S}$ is a strict contraction on $E_w(0,T)\times E_s(0,T)$; 
\[
\norm{\mathcal{S}(\varphi_1,\xi_1)-\mathcal{S}(\varphi_2,\xi_2)}_{E_w(0,T)\times E_s(0,T)}
\le \frac{1}{2}\norm{(\varphi_1-\varphi_2,\xi_1-\xi_2)}_{E_w(0,T)\times E_s(0,T)}.
\]
Therefore, the mapping $\mathcal{S}$ admits a unique fixed point in a closed subset $B_R(T^\ast)$ for some $T^\ast>0$. 
In particular, there exists a unique solution $(w,s)$ to \eqref{eq:w} on $(0,T^\ast)$.
\end{proof}

\begin{rem}\label{rem:reg}
\begin{enumerate}
\item At the level of local well-posedness established in Theorem \ref{thm:w}, no sign condition on $f$ or $\p_x f$ is required.
In particular, we do not characterize the sign of $w$ or $s$. 
The theorem also holds for $\alpha<0$. 
\item The possibility that $s(t)\equiv 0$ is not excluded by the above analysis.
This corresponds to a stationary interface and is consistent with the formulation. 
\item Although the equation
\[\alpha s'(t) =-\p_y w(0,t) - f(s(t),t)\]
is satisfied in the sense of $F_s$, we can in fact improve the regularity.
More precisely, we have
\[\p_y w|_{y=0} \in W^{\theta,p_1}(0,T^\ast)\subset C^{\theta-1/p_1}[0,T^\ast],\quad0<\theta-\frac{1}{p_1}<\frac{1}{2}\Bigl(1-\frac{2}{p_1}-\frac{1}{q}\Bigr),\]
where we recall that the condition $2/p_1+1/q<1$ was assumed.
Here $C^{\theta-1/p_1}[0,T^\ast]$ denotes the H\"older space. 
See, for example, \cite{Lun}. 

On the other hand,
\[s \in E_s(0,T^\ast)\subset C^{1-1/p_2}[0,T^\ast]\subset C^{\theta-1/p_1}[0,T^\ast].\]
Assuming appropriate H\"older continuity of $f$ in space and time, the composition $t\mapsto f(s(t),t)$ belongs to $C^{\theta-1/p_1}[0,T^*]$, since H\"older continuous functions are stable under composition. 
Consequently, the right-hand side of the above equation is H\"older continuous with exponent $\theta-1/p_1$ (and in particular belongs to a space less regular than $C^{1/2}[0,T^\ast]$).

Therefore, the solution $s\in E_s(0,T^\ast)$ enjoys additional regularity, namely,
\[s' \in C^{\theta-1/p_1}[0,T^\ast].\]
In this case, $s'$ admits a H\"older continuous extension up to $t=0$, and the extended function satisfies 
\[s'(0)=\frac{-\p_y w_0|_{y=0} - f(0,0)}{\alpha}.\] 
By continuity near $t=0$, it follows that there exists $\tilde{T}\in(0,T^\ast]$ such that
\begin{align*}
\begin{cases}
s'(t)>0 \quad {\rm for~} t\in[0,\tilde{T}]&\quad{\rm if~}\quad \p_y w_0|_{y=0}+f(0,0)<0,\\
s'(t)<0 \quad {\rm for~} t\in[0,\tilde{T}]&\quad{\rm if~}\quad \p_y w_0|_{y=0}+f(0,0)>0.
\end{cases}
\end{align*}
In particular, the case $s(t)\not\equiv 0$ is admissible, i.e. the free boundary moves for small time.
Moreover, the improved regularity
\[
s'\in C^{\beta}[0,T^\ast]{\rm~with~}0<\beta<\frac{1}{2}\Bigl(1-\frac{2}{p_1}-\frac{1}{q}\Bigr)
\]
will be used in subsequent arguments.
This additional H\"older regularity will be crucial in Sections \ref{sec:reg} and \ref{sec:v}. 
When the regularity of $f$ is worse, the regularity of $s'$ becomes the same as $f$, which is less regular than $C^\beta$. 
This will be clarified in Theorem \ref{thm:identification}. 
\end{enumerate}
\end{rem}

\section{H\"older Regularity for the $v$-Problem}\label{sec:reg}
This section is devoted to investigating further regularity properties of the solution obtained in Section \ref{sec:w}.
We note that we do not directly solve the original problem \eqref{eq:v} whose unknown functions are $v$ and $s$.
Instead, we regard the problem \eqref{eq:v} as a linear system for $v$ with a given free boundary $s$, where $s$ is taken from the solution of the derivative system \eqref{eq:w}.
The role of Schauder theory in this paper is purely linear and is used only to recover classical regularity. 
As a first step, we introduce H\"older spaces and recall some basic properties.

\begin{definition}[Parabolic H\"older spaces]\label{def:holder}
Let $0<\gamma<1/2$ and $T>0$. 
We denote by $C^{2\gamma,\gamma}(\R\times[0,T])$ the space of all bounded functions $f:\R\times[0,T]\to\R$ such that
\[
\norm{f}_{C^{2\gamma,\gamma}(\R\times[0,T])}:= \sup_{(x,t)\in\R\times[0,T]} |f(x,t)|
+ [f]_{C^{2\gamma,\gamma}(\R\times[0,T])} < \infty,
\]
where the parabolic H\"older semi-norm is defined by
\[
[f]_{C^{2\gamma,\gamma}(\R\times[0,T])}:= \sup_{\substack{(x_1,t_1),(x_2,t_2)\in\R\times[0,T]\\(x_1,t_1)\neq(x_2,t_2)}} \frac{|f(x_1,t_1)-f(x_2,t_2)|}{|x_1-x_2|^{2\gamma}+|t_1-t_2|^{\gamma}}.
\]
\end{definition}
For simplicity of notation, we define the spaces on $\R$; their restriction to $\R_+$ will be also used in this paper.
We also use the Lipschitz function space $C^{0,1}[0,T]$ as a time variable function space. 
Details can be found in \cite{Lun}. 

The following lemma plays a crucial role in transferring H\"older regularity from the source term to the moving-coordinate formulation. 
\begin{lem}[Stability under composition]\label{lem:composition}
Let $0<\gamma<1/2$.
Assume that
\[
f \in C^{2\gamma,\gamma}(\R\times[0,T]),
\qquad
s \in C^{0,1}[0,T].
\]
Then the function
\[F(y,t) := f(s(t)+y,t)\]
belongs to $C^{2\gamma,\gamma}(\R\times[0,T])$.
Moreover, there exists a constant $C>0$, depending only on $\norm{s'}_{L^{\infty}(0,T)}$, such that
\[
\norm{F}_{C^{2\gamma,\gamma}(\R\times[0,T])} \le C \norm{f}_{C^{2\gamma,\gamma}(\R\times[0,T])}.
\]
\end{lem}

\begin{proof}
Let $(y_1,t_1),(y_2,t_2)\in\R\times[0,T]$. Then
\begin{align*}
|F(y_1,t_1)-F(y_2,t_2)|
&= |f(s(t_1)+y_1,t_1)-f(s(t_2)+y_2,t_2)| \\
&\le [f]_{C^{2\gamma,\gamma}(\R\times[0,T])}\Big(|s(t_1)+y_1-(s(t_2)+y_2)|^{2\gamma}+ |t_1-t_2|^{\gamma}\Big).
\end{align*}
Using the Lipschitz continuity of $s$, we estimate
\[|s(t_1)+y_1-(s(t_2)+y_2)|\le |y_1-y_2| + C|t_1-t_2|,\]
which yields 
\[|s(t_1)+y_1-(s(t_2)+y_2)|^{2\gamma}\le C\big(|y_1-y_2|^{2\gamma}+|t_1-t_2|^{2\gamma}\big)\]
by concavity of $x\mapsto x^{2\gamma}$. 
Since $0<\gamma<1/2$, we have $|t_1-t_2|^{2\gamma}\le |t_1-t_2|^{\gamma}$ for $|t_1-t_2|\le 1$, and the case $|t_1-t_2|>1$ is absorbed into the boundedness of $f$. 
Consequently,
\[
|F(y_1,t_1)-F(y_2,t_2)| \le C \norm{f}_{C^{2\gamma,\gamma}(\R\times[0,T])}\big(|y_1-y_2|^{2\gamma}+|t_1-t_2|^{\gamma}\big).
\]
Hence, $F\in C^{2\gamma,\gamma}(\R\times[0,T])$ and completes the proof.
\end{proof}

\begin{rem}
The free boundary $s$ obtained in Section \ref{sec:w} satisfies that the derivative $s'$ is H\"older continuous.
Therefore $s$ satisfies the Lipschitz assumption in Lemma \ref{lem:composition}. 
\end{rem}

\begin{thm}[Schauder theory for the $v$-problem]\label{thm:holder-v}
Let $0<\gamma, \beta<1/2$ and $T>0$.
Assume that
\[f \in C^{2\gamma,\gamma}(\R\times[0,T]), \qquad s \in C^{1+\beta}[0,T],\]
and that the initial data $v_0$ satisfies $v_0 \in C^{2+2\gamma}(\R_+)$ and $\p_y v_0|_{y=0} = \alpha$. 
Then the linear problem
\begin{equation}\label{eq:v-holder}
	\begin{cases}
		\p_t v - \partial_y^2 v = s'(t)\partial_y v + f(s(t)+y,t),& y>0,\; 0<t<T,\\
		\p_y v(0,t) = \alpha, & 0<t<T,\\
		v(y,0) = v_0(y), & y>0
	\end{cases}
\end{equation}
admits a unique classical solution
\[v \in C^{2+2\tilde{\gamma},\,1+\tilde{\gamma}}(\R_+\times[0,T]), \]
where $\tilde{\gamma}=\min\{\gamma, \beta\}$. 
\end{thm}

Theorem \ref{thm:holder-v} follows from the standard Schauder theory for parabolic equations with non-homogeneous Neumann boundary conditions on the half-line (see \cite{Lieb}).
Lemma \ref{lem:composition} guarantees that the forcing term $f(s(t)+y,t)$ possesses the required H\"older regularity.
The regularity $\tilde{\gamma}$ of the solution $v$ is determined by the H\"older regularity of the coefficient of the first-order term $s'(t)\p_y v$.

\begin{cor}[higher regularity]\label{cor:hoelder2}
In addition to the assumption of Theorem \ref{thm:holder-v}, we assume 
\[\p_y v_0\in C^{2+2\tilde{\gamma}}(\R_+), \quad \p_x f\in C^{2\tilde{\gamma}, \tilde{\gamma}}(\R\times [0,T]).\]
Then the solution $v$ in Theorem \ref{thm:holder-v} satisfies $\partial_y v\in C^{2+2\tilde{\gamma}, 1+\tilde{\gamma}}(\R_+\times [0,T])$ and 
\begin{equation}\label{eq:v2}
\begin{cases}
	\p_t (\p_y v) - \partial_y^2 (\p_y v) = s'(t)\partial_y (\p_y v) + \p_x f(s(t)+y,t),& y>0,\; 0<t<T,\\
	\p_y v(0,t) = \alpha, & 0<t<T,\\
	\p_y v(y,0) = \p_y v_0(y), & y>0. 
\end{cases}
\end{equation}
\end{cor}
This additional regularity will be crucial in Section \ref{sec:v} to identify the solution obtained by the maximal $L^p$-regularity approach with the classical solution constructed in the present section.
In particular, we will prove that $w=\p_y v$ in the next section. 

\section{Solvability for the $v$-problem and the original problem}\label{sec:v}

\subsection{Identification of $w=\p_y v$}
The key ingredient for this identification is the uniqueness of solutions to the derivative problem in the maximal regularity class.

\begin{thm}\label{thm:identification}
Under the assumptions of Theorem \ref{thm:w}, $0<\gamma<1/2$, $0<\beta<(1/2)(1-2/p_1-1/q)$, $f\in C^{2\gamma, \gamma}(\R\times [0,T])$, the solution $s$ in Theorem \ref{thm:w} satisfies $s\in C^{1+\tilde{\gamma}}[0,T^\ast]$. 

Furthermore, let $w_0\in C^{2+2\tilde{\gamma}}(\R_+), \p_x f\in C^{2\tilde{\gamma}, \tilde{\gamma}}(\R\times [0,T^\ast])$ and define $v_0:=\int_0^y w_0(\tilde{y})d\tilde{y}$. 
Then we are able to take the solution $v$ in Theorem \ref{thm:holder-v} and Corollary \ref{cor:hoelder2}, and we have 
\[w=\p_y v\]
and $\p_y v\in C^{2+2\tilde{\gamma}, 1+\tilde{\gamma}}(\R_+\times [0,T^\ast])$. 
\end{thm}

\begin{proof}
The first part $s\in C^{1+\tilde{\gamma}}[0,T^\ast]$ is in Remark \ref{rem:reg} (3). 
Let $v$ be the solution obtained in Theorem~\ref{thm:holder-v}.
Define $z := \p_y v$.
Then $z$ satisfies, by Corollary \ref{cor:hoelder2}, 
\begin{equation}\label{eq:z}
	\begin{cases}
		\p_t z - \p_y^2 z= s'(t)\p_y z + \p_x f(s(t)+y,t),& y>0,\; 0<t<T^\ast,\\
		z(0,t)=\alpha, & 0<t<T^\ast,\\
		z(y,0)=\p_y v_0(y)=w_0(y), & y>0.
	\end{cases}
\end{equation}

On the other hand, by Theorem~\ref{thm:w}, the function $w$ satisfies the same initial-boundary value problem \eqref{eq:z} in the sense of $L^p$-maximal regularity.
Since
\[\p_x f(s(t)+y,t)\in L^{p_1}(0,T^\ast;L^q(\R_+))\cap C^{2\gamma,\gamma}(\R_+\times[0,T^\ast]),\]
the problem \eqref{eq:z} admits a unique solution in the class
\[W^{1,p_1}(0,T^\ast;L^q(\R_+))\cap L^{p_1}(0,T^\ast;W^{2,q}(\R_+))\]
and $C^{2+2\tilde{\gamma}, 1+\tilde{\gamma}}(\R_+\times [0,T^\ast])$. 
Therefore, by uniqueness, we conclude 
\[w = z = \p_y v \quad \text{in } \R_+\times[0,T^\ast].\]
\end{proof}

\subsection{Existence and uniqueness theorem for the problems}
In this subsection, we conclude the existence and uniqueness of solutions to the equations \eqref{eq:v} and \eqref{eq:u}. 
Since $\p_t v(0,t)=0$ in the classical sense from the equation \eqref{eq:v-holder} and the second equation in \eqref{eq:w} with $w=\p_y v$, the Dirichlet boundary condition $v(0,t)=0$ for all $t\in[0,T^\ast]$ follows from $v_0(0)=\int_0^0 w_0(\tilde{y})d\tilde{y}=0$.

We can now state the main theorem. 
For the equations \eqref{eq:v} and \eqref{eq:u}, we establish the following.  
\begin{thm}[Solvability for the problems on the half-line]\label{thm:main}
Assume $\alpha\neq0$, the condition {\rm (A)}, and $0<\gamma<1/2$, $0<\beta<(1/2)(1-2/p_1-1/q)$. 
\begin{enumerate}[label=\upshape(\roman*)]
\item Forcing term. \\
Let $f:\R\times[0,T]\to\R$ satisfy the condition {\rm (F)} and $f, \p_x f\in C^{2\gamma,\gamma}(\R\times[0,T])$. 
\item Initial data. \\
Let $w_0 \in B^{2(1-1/p_1)}_{q,p_1}(\R_+)\cap C^{2+2\gamma}(\R_+)$, $w_0(0)=\alpha$ satisfy 
\[v_0(y):=\int_0^y w_0(\tilde{y})d\tilde{y}\]
\end{enumerate}
and $v_0$ belongs to $C^{2+2\gamma}(\R_+)$. 
Then there exists $T^\ast>0$ such that the equation \eqref{eq:v} admits a unique local-in-time classical solution $(v,s)$ in the class 
\[v \in C^{2+2\tilde{\gamma},1+\tilde{\gamma}}(\R_+\times[0,T^\ast]),\quad s\in C^{1+\tilde{\gamma}}[0,T^\ast]\]
and 
\begin{align*}
\p_y v &\in W^{1,p_1}(0,T^\ast;L^q(\R_+)) \cap L^{p_1}(0,T^\ast;W^{2,q}(\R_+))\cap C^{2+2\tilde{\gamma},1+\tilde{\gamma}}(\R_+\times[0,T^\ast]),\\
s&\in W^{1,p_2}(0,T^\ast), 
\end{align*}
where $\tilde{\gamma}=\min\{\gamma, \beta\}$. 

Furthermore, let $u(x,t)=v(x-s(t),t)$ and $u_0(x)=v_0(x)$. 
Then the function $u$ is the unique classical solution of the equation \eqref{eq:u}. 
\end{thm}

Theorem \ref{thm:main} provides a local-in-time classical solution to a combustion-type free boundary problem with the prescribed boundary flux. 
The derivative formulation developed here is expected to be applicable to a broader class of free boundary problems with similar structural features. 

\begin{rem}[Positivity of solutions]\label{rem:positivity}
In addition to the hypotheses of Theorem \ref{thm:main}, assume 
\[
u_0(x)\ge0 \quad (x>0), \qquad
f(x,t)\ge0 \quad (x\in\R,\ 0\le t\le T^*),
\]
and 
\[\alpha>0.\]
Then the corresponding local-in-time classical solution $(u,s)$ obtained in
Theorem \ref{thm:main} satisfies
\[u(x,t)>0 \qquad (x>s(t),\ 0< t\le T^\ast)\]
by the standard parabolic maximum principle together with Hopf's boundary point lemma. 
See, for example, \cite{Fried, Lieb, PS0, PW}. 

If $\alpha<0$, the above positivity property may fail in general, even if $u_0\ge0$ and $f\ge0$. 
In this case, the local well-posedness result of Theorem \ref{thm:main} remains valid, but the standard maximum principle argument does not directly apply.
\end{rem}

\begin{rem}
As we note in Remark \ref{rem:reg}, the case $s(t)\equiv0$ is not excluded. 
However, it is possible to show $s(t)\not\equiv0$.  
In the case with $\alpha>0$ and $f\equiv0$, the assumption $\partial_{yy} v_0(0)>0$ implies that $s'(t)<0$ for sufficiently small $t>0$. 
\end{rem}

\section{The case of bounded interval with two free boundaries}\label{sec:bounded}
In contrast to the half-line case, the bounded interval setting allows two free boundaries and a time-dependent spatial scaling.
As a consequence, additional lower-order terms and coupling effects appear in the transformed system, which require a careful calculation developed in Sections \ref{sec:reformation}--\ref{sec:v}. 
The crucial difference from the half-line case is that on bounded intervals the spatial integrability of $w$ allows us to recover $v$ by indefinite integration, without relying on additional H\"older regularity. 
This simplification does not extend to unbounded intervals unless additional decay assumptions are imposed. 

\subsection{Formulation}
\label{subsection:bounded-formulation}

Let $s_{\pm,0}\in \R$ with $s_{-,0}<s_{+,0}$, and $s_\pm(t)$ be two unknown functions such that
\[s_-(t) < s_+(t), \qquad t \ge 0\]
and $s_\pm(0)=s_{\pm,0}$. 
We are concerned with the following free boundary problem:
\begin{equation}\label{eq:u-bounded}
	\begin{cases}
		\partial_t u - \partial_x^2 u = f(x,t),& s_-(t) < x < s_+(t),\ t>0,\\
		u(s_\pm(t),t)=0,& t>0,\\
		\partial_\nu u = -\alpha_\pm, &x=s_\pm(t),\ t>0,\\
		u(x,0)=u_0(x),& s_{-,0}<x<s_{+,0}, \\
		s_\pm(0)=s_{\pm,0}.&
	\end{cases}
\end{equation}
Here $\alpha_\pm>0$ are given constants.
Since the outward unit normal satisfies $\nu=-1$ at $x=s_-(t)$ and $\nu=+1$ at $x=s_+(t)$, the Neumann boundary conditions can be written as
\[
\partial_x u(s_-(t),t)=\alpha_-,
\qquad
\partial_x u(s_+(t),t)=-\alpha_+.
\]
Note that the signs of the boundary fluxes differ due to the orientation of outward normal vectors. 

Differentiating the Dirichlet boundary conditions
$u(s_\pm(t),t)=0$ with respect to time, we obtain
\[
\p_t u + s_\pm'(t)\p_x u = 0
\quad \text{at } x=s_\pm(t).
\]
By using the equation $\p_t u=\p_x^2 u+f$ and Neumann boundary conditions, this yields
\begin{equation}\label{eq:bc}
	\begin{cases}
		\alpha_- s_-'(t)= -\partial_{xx}u(s_-(t),t)-f(s_-(t),t), \\
		\alpha_+ s_+'(t)= \partial_{xx}u(s_+(t),t)+f(s_+(t),t).
	\end{cases}
\end{equation}

To fix the spatial domain, we introduce the change of variables
\[
y := \frac{x-s_-(t)}{s_+(t)-s_-(t)}, \qquad
v(y,t) := u\bigl((1-y)s_-(t)+y s_+(t),t\bigr),
\]
which maps the moving interval $(s_-(t),s_+(t))$ onto the fixed interval $(0,1)$.

A direct computation shows that
\[
\p_y v = (s_+(t)-s_-(t))\p_x u,
\qquad
\p_y^2 v = (s_+(t)-s_-(t))^2\p_x^2 u,
\]
and
\[\p_t v = \p_t u + ((1-y)s_-'(t)+y s_+'(t))\p_x u\quad(\Leftrightarrow \partial_t u= \p_t v - \frac{(1-y)s_-'(t)+y s_+'(t)}{s_+(t)-s_-(t)}\p_y v).\]

Therefore, the pair $(v,s_\pm)$ satisfies
\begin{equation}\label{eq:v-bounded}
	\begin{cases}
		\p_t v - (s_+(t)-s_-(t))^{-2}\p_y^2 v = \dfrac{(1-y)s_-'(t)+y s_+'(t)}{s_+(t)-s_-(t)}\p_y v \\
		\hspace{45mm}+ f\bigl((1-y)s_-(t)+y s_+(t),t\bigr),& 0<y<1,\ t>0,\\
		\alpha_- s_-'(t) = -(s_+(t)-s_-(t))^{-2}\p_y^2 v(0,t) -f(s_-(t),t),&t>0, \\
		\alpha_+ s_+'(t) = (s_+(t)-s_-(t))^{-2}\p_y^2 v(1,t) +f(s_+(t),t),&t>0, \\
		\p_y v(0,t)=\alpha_- (s_+(t)-s_-(t)), &t>0, \\
		\p_y v(1,t)=-\alpha_+ (s_+(t)-s_-(t)),& t>0,\\
		v(y,0)=v_0(y),& 0<y<1, \\
		s_\pm(0)=s_{\pm,0},&
	\end{cases}
\end{equation}
where we have set $v_0(y)=u_0((1-y)s_{-,0}+y s_{+,0})$. 
At a formal level, taking the trace at $y=0$ and $y=1$ yields $\p_t v(0,t)=\p_t v(1,t)=0$, which suggests that $v(0,t)=v(1,t)=0$ provided $v_0(0)=v_0(1)=0$.
This will be justified a posteriori once sufficient regularity is established.
Throughout this section, we assume $s_+(t)-s_-(t)$ remains strictly positive on a short time interval $[0,T]$, which will be guaranteed by the initial condition $s_{-,0}<s_{+,0}$ and continuity. 

\subsection{Derivative formulation}
\label{subsection:bounded-derivative}

As in the half-line case, the main difficulty in the system \eqref{eq:v-bounded} lies in the appearance of the second-order boundary traces $\p_y^2 v$.
These terms are not directly compatible with the standard maximal $L^p$-$L^q$ regularity theory.

To overcome this difficulty, we introduce a derivative formulation.
Set
\[w := \partial_y v.\]

Formally differentiating \eqref{eq:v-bounded} with respect to $y$, we obtain
\begin{equation}\label{eq:w-bounded}
	\begin{cases}
		\p_t w - (s_+(t)-s_-(t))^{-2}\p_y^2 w\\
		\hspace{10mm}= \p_y \biggl(\dfrac{(1-y)s_-'(t)+y s_+'(t)}{s_+(t)-s_-(t)} w\biggr) \\
		\hspace{15mm}+ (s_+(t)-s_-(t))\p_x f\bigl((1-y)s_-(t)+y s_+(t),t\bigr), & 0<y<1,\ t>0,\\
		\alpha_- s_-'(t) = -(s_+(t)-s_-(t))^{-2}\p_y w(0,t) -f(s_-(t),t),&t>0, \\
		\alpha_+ s_+'(t) = (s_+(t)-s_-(t))^{-2}\p_y w(1,t) +f(s_+(t),t),&t>0, \\
		w(0,t)=\alpha_- (s_+(t)-s_-(t)),&t>0, \\
		w(1,t)=-\alpha_+ (s_+(t)-s_-(t)),& t>0,\\
		w(y,0)=w_0(y)(:=\p_y v_0(y)),& 0<y<1, \\
		s_\pm(0)=s_{\pm,0}.&
	\end{cases}
\end{equation}
Here, we see 
\[\p_y \biggl(\frac{(1-y)s_-'(t)+y s_+'(t)}{s_+(t)-s_-(t)} w\biggr)=\frac{(1-y)s_-'(t)+y s_+'(t)}{s_+(t)-s_-(t)} \p_y w + \frac{s_+'(t)-s_-'(t)}{s_+(t)-s_-(t)}w.\] 

Observe that, at this point, the equivalence between \eqref{eq:v-bounded} and the derivative system \eqref{eq:w-bounded} has not yet been justified.
However, in contrast to the half-line case, the identification $w=\p_y v$ can be established on bounded intervals by means of indefinite integration. 
This argument relies on the integrability of $w$ over the spatial interval, which is available on the bounded interval but not in the half-line setting.

\subsection{Solvability for the $w$-problem \eqref{eq:w-bounded} and the original problem}
Let $1<p_1,p_2,q<\infty$ and $I:=(0,1)$. 
We define
\begin{align*}
F^b_w&:=L^{p_1}(0,T;L^q(I)), \\
E^b_w&:=W^{1,p_1}(0,T;L^q(I))\cap L^{p_1}(0,T;W^{2,q}(I)), \\
X^b_w&:=B^{2(1-1/p_1)}_{q,p_1}(I). 
\end{align*}

The estimates below are bounded-interval counterparts of Proposition \ref{prop:product}. 
The key point is that the denominator $\xi_+(t)-\xi_-(t)$ is uniformly bounded away from zero for sufficiently small time. 

\begin{prop}[Product and trace estimates]\label{prop:product_b}
Assume the condition {\rm (A)}. 
Then there exist $T_0>0$ and $C>0$, depending only on $p_1, p_2, q$ and the length of the initial spatial interval, such that the following estimates hold for all $T\in(0,T_0]$. 
\begin{enumerate}[label=\upshape(\roman*)]
\item Product estimates. \\
Let $\varphi\in E^b_w$ and $\xi_\pm\in E_s$ with $\xi_-(0)<\xi_+(0)$. 
Then 
\begin{align*}
\norm{\frac{(1-y)\xi_-'(t)+y \xi_+'(t)}{\xi_+(t)-\xi_-(t)}\p_y\varphi}_{F^b_w}
	&\le C T^{\frac{1}{p_1}-\frac{1}{p_2}}\bigl(\norm{\varphi}_{E^b_w}+\norm{\varphi|_{t=0}}_{X^b_w}\bigr)(\norm{\xi_+}_{E_s}+\norm{\xi_-}_{E_s}),\\
\norm{\frac{\xi_+'(t)-\xi_-'(t)}{\xi_+(t)-\xi_-(t)}w}_{F^b_w}
	&\le C T^{\frac{1}{p_1}-\frac{1}{p_2}}\bigl(\norm{\varphi}_{E^b_w}+\norm{\varphi|_{t=0}}_{X^b_w}\bigr)(\norm{\xi_+}_{E_s}+\norm{\xi_-}_{E_s}),\\
\norm{(\xi_+(t)-\xi_-(t))^{-2}\p_y\varphi|_{y=0}}_{F_s}
	&\le C T^{\frac{1}{p_2}}\bigl(\norm{\varphi}_{E^b_w}+\norm{\varphi|_{t=0}}_{X^b_w}\bigr), \\
	\norm{(\xi_+(t)-\xi_-(t))^{-2}\p_y\varphi|_{y=1}}_{F_s}
	&\le C T^{\frac{1}{p_2}}\bigl(\norm{\varphi}_{E^b_w}+\norm{\varphi|_{t=0}}_{X^b_w}\bigr). 
\end{align*}
 \item Difference estimates (contraction properties).\\
Let $\varphi_1, \varphi_2\in E^b_w$ with $\varphi_1|_{t=0}=\varphi_2|_{t=0}$ and $\xi_{1,\pm}, \xi_{2,\pm}\in E_s$ with $\xi_{1,-}(0)<\xi_{1,+}(0)$, $\xi_{1,\pm}(0)=\xi_{2,\pm}(0)$. 
Then 
\begin{align*}
&\norm{\frac{(1-y)\xi_{1,-}'(t)+y \xi_{1,+}'(t)}{\xi_{1,+}(t)-\xi_{1,-}(t)}\p_y\varphi_1-\frac{(1-y)\xi_{2,-}'(t)+y \xi_{2,+}'(t)}{\xi_{2,+}(t)-\xi_{2,-}(t)}\p_y\varphi_2}_{F^b_w}\\
	&\le C T^{\frac{1}{p_1}-\frac{1}{p_2}}\biggl(\bigl(\norm{\varphi_1}_{E^b_w}+\norm{\varphi_1|_{t=0}}_{X^b_w}\bigr)(\norm{\xi_{1,+}-\xi_{2,+}}_{E_s}+\norm{\xi_{1,-}-\xi_{2,-}}_{E_s}) \\
	&\qquad + \norm{\varphi_1-\varphi_2}_{E^b_w}(\norm{\xi_{2,+}}_{E_s}+\norm{\xi_{2,-}}_{E_s})\biggr),\\
&\norm{\frac{\xi_{1,+}'(t)-\xi_{1,-}'(t)}{\xi_{1,+}(t)-\xi_{1,-}(t)}\varphi_1 - \frac{\xi_{2,+}'(t)-\xi_{2,-}'(t)}{\xi_{2,+}(t)-\xi_{2,-}(t)}\varphi_2}_{F^b_w}\\
	& \le C T^{\frac{1}{p_1}-\frac{1}{p_2}}\biggl(\bigl(\norm{\varphi_1}_{E^b_w}+\norm{\varphi_1|_{t=0}}_{X^b_w}\bigr)(\norm{\xi_{1,+}-\xi_{2,+}}_{E_s}+\norm{\xi_{1,-}-\xi_{2,-}}_{E_s})\\
	&\qquad +\norm{\varphi_1-\varphi_2}_{E^b_w}(\norm{\xi_{2,+}}_{E_s}+\norm{\xi_{2,-}}_{E_s})\biggr),\\
&\norm{(\xi_{1,+}(t)-\xi_{1,-}(t))^{-2}\p_y\varphi_1|_{y=0} - (\xi_{2,+}(t)-\xi_{2,-}(t))^{-2}\p_y\varphi_2|_{y=0}}_{F_s}\\
	&\le C T^{\frac{1}{p_2}}\biggl(\norm{\varphi_1-\varphi_2}_{E^b_w}+(\norm{\varphi_1}_{E^b_w}+\norm{\varphi_1|_{t=0}}_{X^b_w})(\norm{\xi_{1,+}-\xi_{2,+}}_{E_s}+\norm{\xi_{1,-}-\xi_{2,-}}_{E_s})\biggr), \\
&\norm{(\xi_{1,+}(t)-\xi_{1,-}(t))^{-2}\p_y\varphi_1|_{y=1} - (\xi_{2,+}(t)-\xi_{2,-}(t))^{-2}\p_y\varphi_2|_{y=1}}_{F_s}\\
	&\le C T^{\frac{1}{p_2}}\biggl(\norm{\varphi_1-\varphi_2}_{E^b_w}+(\norm{\varphi_1}_{E^b_w}+\norm{\varphi_1|_{t=0}}_{X^b_w})(\norm{\xi_{1,+}-\xi_{2,+}}_{E_s}+\norm{\xi_{1,-}-\xi_{2,-}}_{E_s})\biggr).
\end{align*}
\end{enumerate}
\end{prop}

\begin{proof}
The proof is parallel to that of Proposition~3.1, and we only indicate the
necessary modifications.

We may assume that $\xi_\pm \in E_s$ are given functions such that
\[0 < \underline{L} \le \xi_+(t)-\xi_-(t) \le \overline{L} \quad \text{for all } t\in[0,T],\]
for some constants $\underline{L},\overline{L}>0$ and for some $T>0$ because of $\xi_-(0)<\xi_+(0)$ and the embedding $E_s\subset C[0,T]$.
Then the coefficient
\[\frac{1}{\xi_+(t)-\xi_-(t)}\]
is uniformly positive and bounded on $(0,T)$.

Moreover, the coefficients
\[\frac{(1-y)\xi_-'(t)+y \xi_+'(t)}{\xi_+(t)-\xi_-(t)}, \quad \frac{\xi_+'(t)-\xi_-'(t)}{\xi_+(t)-\xi_-(t)}\]
belong to
\[L^{p_2}\bigl(0,T;W^{1,\infty}(I)\bigr),\]
with bounds depending only on $\underline{L}$, $\overline{L}$, and $\norm{\xi_\pm'}_{L^p(0,T)}$.
The difference estimates are obtained in the same way as in Proposition \ref{prop:product}, using the Lipschitz continuity of the maps 
\[(\xi_+, \xi_-)\mapsto (\xi_+(t)-\xi_-(t))^{-1}\]
on the set $\{\xi_+(t)-\xi_-(t)\ge \underline{L}\}$. 
\end{proof}

\begin{thm}[Local existence for $(w,s_\pm)$]\label{thm:w_b}
Assume the conditions {\rm (A)} and {\rm (F)}. 
Let 
\[w_0\in X^b_w, \qquad w_0(0)=\alpha_-(s_{+,0}-s_{-,0}), \quad w_0(1)=-\alpha_+(s_{+,0}-s_{-,0}),\]
with $s_{-,0}<s_{+,0}$. 
Then there exists $T^\ast > 0$ such that the derivative system \eqref{eq:w-bounded} admits a unique local-in-time solution in the maximal regularity class 
\[(w, s_\pm)\in E^b_w(0,T^\ast)\times E_s(0,T^\ast)^2\]
satisfying $s_+(t)-s_-(t)>0$ for all  $t\in[0,T^\ast]$ and 
\[\norm{w}_{E^b_w(0,T^\ast)} + \norm{s_+}_{E_s(0,T^\ast)} + \norm{s_-}_{E_s(0,T^\ast)}\leq  C_{T^\ast, f, \alpha_\pm} + C\norm{w_0}_{X^b_w}\]
for some constants $C_{T^\ast, f, \alpha_\pm}$ and $C$. 
\end{thm}

\begin{proof}
The proof is based on the same strategy as Theorem \ref{thm:w}, using the uniform positivity of $s_+(t)-s_-(t)$.
The positivity of $s_+(t)-s_-(t)$ follows from the continuity of $s_\pm\in E_s(0,T)\subset C[0,T]$ and the fact that $s_{-,0}<s_{+,0}$. 
We sketch the argument and indicate the necessary modifications.

We consider the closed subset 
\begin{align*}
&B_R^b(T) \\
&:= \left\{(w,s_\pm)\in E_w^b(0,T)\times E_s(0,T)^2~\middle|~
  \begin{aligned} 
&w_0(0)=\alpha_-(s_{+,0}-s_{-,0}),\\
&w_0(1)=-\alpha_+(s_{+,0}-s_{-,0}),\\
&w|_{t=0}=w_0,\ s_\pm(0)=s_{\pm,0},\\
&\norm{w}_{E_w^b(0,T)}+\norm{s_+}_{E_s(0,T)}+\norm{s_-}_{E_s(0,T)}\le R
  \end{aligned}
  \right\}
\end{align*}
endowed with the natural norm, where $T>0$ and $R>0$ are fixed later.

We define the solution mapping 
\begin{align*}
\begin{array}{rccc}
\mathcal{S}^b:&B_R^b(T)&\longrightarrow&E_w^b(0,T) \times E_s(0,T)^2\\
        & \rotatebox{90}{$\in$}&               & \rotatebox{90}{$\in$} \\
        &(\varphi,\xi_\pm)& \longmapsto   & (w,s_\pm)
\end{array}
\end{align*}
to
\begin{equation*}
	\begin{cases}
		\p_t w - (s_{+,0}-s_{-,0})^{-2}\p_y^2 w \\
		\quad =\biggl((\xi_+(t)-\xi_-(t))^{-2}-(s_{+,0}-s_{-,0})^{-2}\biggr)\p_y^2 \varphi  &\\
			\qquad + \p_y \biggl(\dfrac{(1-y)\xi_-'(t)+y \xi_+'(t)}{\xi_+(t)-\xi_-(t)} \varphi\biggr)&\\
			\qquad + (\xi_+(t)-\xi_-(t))\p_x f\bigl((1-y)\xi_-(t)+y \xi_+(t),t\bigr), & 0<y<1,\ t>0,\\
		\alpha_- s_-'(t) = -(\xi_+(t)-\xi_-(t))^{-2}\p_y \varphi(0,t) -f(\xi_-(t),t)&t>0, \\
		\alpha_+ s_+'(t) = (\xi_+(t)-\xi_-(t))^{-2}\p_y \varphi(1,t) +f(\xi_+(t),t)&t>0, \\
		w(0,t)=\alpha_- (\xi_+(t)-\xi_-(t)), &t>0, \\
		 w(1,t)=-\alpha_+ (\xi_+(t)-\xi_-(t)),& t>0,\\
		w(y,0)=w_0(y),& 0<y<1, \\
		s_\pm(0)=s_{\pm,0}. &
	\end{cases}
\end{equation*}
We treat the linear system with frozen coefficients. 
In addition to Proposition \ref{prop:product_b}, we use the following estimates: 
\begin{align*}
\norm{\biggl((\xi_+(t)-\xi_-(t))^{-2}-(s_{+,0}-s_{-,0})^{-2}\biggr)\p_y^2 \varphi}_{F_w^b}&\le  CT^{1-\frac{1}{p_2}}(\norm{\xi_+}_{E_s}+\norm{\xi_-}_{E_s})\norm{\varphi}_{E_w^b}, \\
\norm{(\xi_+(t)-\xi_-(t))\p_x f\bigl((1-y)\xi_-(t)+y \xi_+(t),t\bigr)}_{F_w^b}&\le C\norm{\p_x f}_{L^{p_1}(0,T; L^q(\R))}, \\
\norm{\alpha_\pm (\xi_+(t)-\xi_-(t))}_{E_w^b}&\le C|\alpha_\pm| T^{\frac{1}{p_1}-\frac{1}{p_2}}\norm{\xi_+-\xi_-}_{E_s}, 
\end{align*}
where the first estimate comes from 
\begin{align*}
\left|(\xi_+(t)-\xi_-(t))^{-2}-(s_{+,0}-s_{-,0})^{-2}\right| &\le C(|\xi_+(t)-s_{+,0}|+|\xi_-(t)-s_{-,0}|)\\
&\le CT^{1-\frac{1}{p_2}}(\norm{\xi_+}_{E_s}+\norm{\xi_-}_{E_s})
\end{align*}
and the third estimate comes from the H\"older inequality. 
Collecting the above estimates, we can prove $\mathcal{S}^b$ is the self-map for sufficiently small $T>0$. 
Similarly, we can prove the contraction on $B_R^b(T)$.
The fixed point $(w,s_\pm)\in B_R^b(T)$ satisfies the equation \eqref{eq:w-bounded}. 
This completes the proof. 
\end{proof}
The proof shows that the bounded interval case differs from the half-line case only by the presence of additional lower-order terms and some scaling terms $\xi_+(t)-\xi_-(t)$, which are controlled by the same maximal regularity framework.

We define
\[v(y,t):=\int_0^y w(\tilde{y},t)\,d\tilde{y}.\]
We have the identification $w=\partial_y v$ and $v(0,t)=0$. 
Moreover, from the equation \eqref{eq:w-bounded}, we see  
\[\frac{d}{dt}\int_0^1 w(\tilde{y},t)d\tilde{y}=0.\]
Therefore, if we assume $\int_0^1 w_0(\tilde{y})d\tilde{y}=0$, then $v(1,t)=\int_0^1 w(\tilde{y},t)d\tilde{y}=\int_0^1 w_0(\tilde{y})d\tilde{y}=0$. 
Thus, we prove the following theorem on the bounded interval. 

\begin{thm}[Solvability for the problems on the bounded interval]\label{thm:main2}
Assume that $\alpha_\pm\neq0$ and the conditions {\rm (A)}, {\rm (F)} hold. 
Let 
\begin{align*}
&w_0 \in B^{2(1-1/p_1)}_{q,p_1}(I),\quad \int_0^1 w_0(\tilde{y})d\tilde{y}=0, \\
&w_0(0)=\alpha_-(s_{+,0}-s_{-,0}), \quad w_0(1)=-\alpha_+(s_{+,0}-s_{-,0})
\end{align*}
with $s_{-,0}<s_{+,0}$ and define
\[v_0(y):=\int_0^y w_0(\tilde{y})d\tilde{y}. \]
Then there exists $T^\ast>0$ such that the equation \eqref{eq:v-bounded} admits a unique local-in-time solution $(v,s)$ in the maximal regularity class 
\begin{align*}
v &\in W^{1,p_1}(0,T^\ast;W^{1,q}(I)) \cap L^{p_1}(0,T^\ast;W^{3,q}(I)), \\
s &\in W^{1,p_2}(0,T^\ast). 
\end{align*}

Furthermore, let $u(x,t)=v\Bigl(\dfrac{x-s_-(t)}{s_+(t)-s_-(t)},t\Bigr)$ and $u_0(x)=v_0\Bigl(\dfrac{x-s_{-,0}}{s_{+,0}-s_{-,0}}\Bigr)$. 
Then the function $u$ is the unique solution of the equation \eqref{eq:u-bounded}. 
\end{thm}

\begin{rem}[Positivity of solutions]\label{rem:positivity2}
Under the assumptions 
\[
u_0(x)\ge0 \quad (s_{-,0}<x<s_{+,0}), \qquad
f(x,t)\ge0 \quad (x\in\R,\ 0\le t\le T^*),
\]
and 
\[\alpha_\pm>0,\]
the positivity of the solution $u$ in Theorem \ref{thm:main2} follows from the parabolic maximum principle together with Hopf's boundary point lemma, as in Remark \ref{rem:positivity}. 
\end{rem}

\paragraph{Acknowledgements}
The work of the first author was partially supported by JSPS Grant-in-Aid for Early-Career Scientists JP22K13948. 
The work of the second author was partly supported by JSPS KAKENHI Grant Numbers JP19H00639, JP20K20342, JP24K00531 and JP24H00183 and by Arithmer Inc., Daikin Industries, Ltd.\ and Ebara Corporation through collaborative grants. 
The work of the third author was partially supported by JSPS KAKENHI Grant Number JP20K14350.

\section*{Declarations}
\textbf{Data Availability}
No data was used for the research described in the article. \\
\textbf{Conflict of interest}
The authors declare that they have no conflict of interest.

\bigskip

\noindent
(K.~Furukawa)
{\it Email address}: \href{mailto:furukawa@sci.u-toyama.ac.jp}{\nolinkurl{furukawa@sci.u-toyama.ac.jp}}\\
(Y.~Giga)
{\it Email address}: \href{mailto:labgiga@ms.u-tokyo.ac.jp}{\nolinkurl{labgiga@ms.u-tokyo.ac.jp}}\\
(N.~Kajiwara)
{\it Email address}: \href{mailto:kajiwara.naoto.p4@f.gifu-u.ac.jp}{\nolinkurl{kajiwara.naoto.p4@f.gifu-u.ac.jp}}

\end{document}